\documentclass[12pt,leqno]{amsart}
\usepackage[a4paper,left=2.5cm,right=2.5cm,top=2cm,bottom=2cm]{geometry}
\AtEndDocument{\vfill\eject\batchmode}% Suppress font verbiage at end of output
\usepackage{amssymb,mathrsfs,mathtools}
\usepackage[PostScript=dvips,nohug]{diagrams}
\makeatletter
\theoremstyle{plain}

\newtheorem{thm}{Theorem}[section]
\newtheorem{prop}[thm]{Proposition}
\newtheorem{cor}[thm]{Corollary}
\newtheorem{lemma}[thm]{Lemma}

\theoremstyle{definition}
\newtheorem{defn}[thm]{Definition}

\theoremstyle{remark}
\newtheorem{rem}[thm]{Remark}
\newtheorem{rems}[thm]{Remarks}

\newtheorem{example}{Example}[section]
\newtheorem{examples}[example]{Examples}
\numberwithin{equation}{section}

\newcounter{numl}

\newcommand{\labelnuml}{\textup{(\arabic{numl})}}
\newenvironment{numlist}{\begin{list}{\labelnuml}%
{\usecounter{numl}\setlength{\leftmargin}{0pt}%
\advance\@listdepth1\relax%
\setlength{\itemindent}{2\parindent}%
\setlength{\itemsep}{\smallskipamount}% 0pt or \smallskipamount
\def\makelabel ##1{\hss \llap {\upshape ##1}}}}{\advance\@listdepth-1\relax%
\end{list}}
\newenvironment{bulletlist}{\begin{list}{\labelitemi}%
{\setlength{\leftmargin}{\parindent}%
\advance\@listdepth1\relax%
\def\makelabel ##1{\hss \llap {\upshape ##1}}}}{\advance\@listdepth-1\relax%
\end{list}}
\newcommand{\acknowledge}{\subsection*{Acknowledgements}}
\newcommand{\thismonth}{\ifcase\month\or
  January\or February\or March\or April\or May\or June\or
  July\or August\or September\or October\or November\or December\fi
  \space\number\year}
\renewcommand{\@secnumfont}{\relax}
\newcommand{\low}{\@ifnextchar^{}{^{\vphantom x}}}
\newcommand{\high}{\@ifnextchar_{}{_{\vphantom I}}}
\DeclareSymbolFont{script}{U}{eus}{m}{n}
\DeclareMathSymbol{\EuWedge}{0}{script}{"5E}
\DeclareMathAlphabet{\mathrmsl}{OT1}{cmr}{m}{sl}
\newcommand{\rssymb}[2]{\newcommand{#1}{{\mathrmsl{#2}}}}
\newcommand{\calsymb}[2]{\newcommand{#1}{{\mathcal{#2}}}}
\newcommand{\bbsymb}[2]{\newcommand{#1}{{\mathbb{#2}}}}
\newcommand{\liealg}[2]{\newcommand{#1}{{\mathfrak{#2}}\low}}
\newcommand{\liealr}[2]{\renewcommand{#1}{{\mathfrak{#2}}\low}}
\newcommand{\lieoper}[2]{\newcommand{#1}{\mathop
  {\mathfrak{#2}\null}\nolimits}}
\newcommand{\oper}[3][n]{\newcommand{#2}{\mathop
  {\mathrm{#3}\null}\ifx n#1\nolimits\else\limits\fi}}
\newcommand{\rsoper}[3][n]{\newcommand{#2}{\mathop
  {\mathrmsl{#3}\null}\ifx n#1\nolimits\else\limits\fi}}
\bbsymb\C{C} \bbsymb\F{F} \bbsymb\HQ{H}\bbsymb\N{N} \bbsymb\Q{Q}
\bbsymb\R{R} \bbsymb\U{U} \bbsymb\V{V} \bbsymb\W{W} \bbsymb\Z{Z}
\calsymb\cA{A} \calsymb\cB{B} \calsymb\cC{C} \calsymb\cD{D} \calsymb\cE{E}
\calsymb\cF{F} \calsymb\cG{G} \calsymb\cH{H} \calsymb\cI{I} \calsymb\cJ{J}
\calsymb\cK{K} \calsymb\cL{L} \calsymb\cM{M} \calsymb\cN{N} \calsymb\cO{O}
\calsymb\cP{P} \calsymb\cQ{Q} \calsymb\cR{R} \calsymb\cS{S} \calsymb\cT{T}
\calsymb\cU{U} \calsymb\cV{V} \calsymb\cW{W} \calsymb\cX{X} \calsymb\cY{Y}
\calsymb\cZ{Z}
\liealg\g{g} \liealg\h{h} \liealg\p{p} \liealg\q{q} \liealr\t{t}
\newcommand{\eps}{\varepsilon}
\newcommand{\gam}{\gamma} 
\newcommand{\lam}{\lambda}
\newcommand{\Sig}{{\mathrmsl\Sigma}}

\newcommand{\Om}{{\mathrmsl\Omega}}
\renewcommand{\geq}{\geqslant} \renewcommand{\leq}{\leqslant}
\rsoper\End{End} \rsoper\Hom{Hom}                % Vector space constructions
\rsoper\Sym{Sym} \rsoper\Skew{Skew}
\rsoper\Aut{Aut} \rsoper\SAut{SAut}              % Group constructions
\rsoper\GL{GL}\rsoper\SL{SL}\rsoper\PGL{PGL}\rsoper\PSL{PSL}\rsoper\Symp{Sp}
\rsoper\CO{CO}\rsoper\On{O} \rsoper\SO{SO}  \rsoper\Pin{Pin}\rsoper\Spin{Spin}
\rsoper\CU{CU}\rsoper\Un{U} \rsoper\SU{SU}
\rsoper\Diff{Diff} \rsoper\SDiff{SDiff}
\lieoper\der{der}                                % Lie algebra constructions
\lieoper\gl{gl} \lieoper\sgl{sl}\lieoper\symp{sp}
\lieoper\co{co} \lieoper\so{so} \lieoper\spin{spin}
\lieoper\cu{cu} \lieoper\un{u}  \lieoper\su{su}
\rsoper\Vect{Vect} \rsoper\Ham{Ham}
\rsoper\Stab{Stab} \lieoper\stab{stab} %stabilizer
\oper\real{Re}  %real part
\newcommand{\ip}[1]{\langle#1\rangle}

\newcommand{\norm}[2][]{|\mkern-2mu|#2|\mkern-2mu|
  _{\lower1pt\hbox{${}_{#1}$}}}
\newcommand{\Norm}[2][]{\bigl|\mkern-3mu\bigr|#2\bigr|\mkern-3mu\bigr|
  _{\lower1pt\hbox{${}_{#1}$}}}

\newcommand{\Dsum}{\bigoplus}               % big direct sum
\newcommand{\tens}{\otimes}                 % small tensor product
\newcommand{\Wedge}{\EuWedge}               % wedge product functor
\newcommand{\del}{\partial}                 % directional derivative
\newcommand{\Gr}{\mathrmsl{Gr}}             % grassmannian
\newcommand{\Proj}{\mathrmsl{P}}            % projective
\newcommand{\st}{\mathrel{|}}               % set theoretic such that
\rsoper\dimn{dim}                           % dimension
\rsoper\kernel{ker}\rsoper\image{im}        % kernel and image
\rsoper\alt{alt}   \rsoper\sym{sym}         % alternating and symmetric part
\rsoper\Ad{Ad}     \rsoper\ad{ad}           % adjoint action or bundle
\rsoper\CoAd{CoAd} \rsoper\coad{coad}       % coadjoint action
\rsoper\trace{tr}  \rsoper\trfree{tf}       % trace and tracefree part
\rsoper\detm{det}                           % determinant
\rsoper\Vol{Vol}                            % volume
\rsoper\divg{div}                           % divergence
\renewcommand{\d}{{\mathrmsl{d}}}           % exterior derivative
\rssymb\iden{id}                            % identity
\rssymb\vol{vol}                            % volume element
\makeatother
\newcommand{\interior}{\mathinner\lrcorner} % interior multiplication
\newcommand{\sub}{\subseteq}
\newcommand{\vspn}[1]{\left<#1\right>}
\rsoper{\pr}{pr}
\newcommand{\bp}{{\boldsymbol p}}

\newcommand{\bs}{{\boldsymbol s}}
\newcommand{\bt}{{\boldsymbol t}}
\newcommand{\bw}{{\boldsymbol w}}
\newcommand{\Lb}{\cL}% line bundle
\newcommand{\cod}{{\mathscr U}}% codomain for u
\newcommand{\cov}{{\mathscr V}}% codomain for R
\newcommand{\cow}{{\mathscr W}}% codomain for grassmannian systems
\newcommand{\as}{{\mathscr M}}% affine domain
\newcommand{\ned}{{\mathscr N}}% net domain
\newcommand{\uk}{{\boldsymbol u}}% general unknown
\newcommand{\bR}{{\boldsymbol R}}% hydrodynamic unknown
\newcommand{\br}{{\boldsymbol r}}% hydrodynamic coordinate
\newcommand{\Uk}{U}% hydrodynamic reduction
\newcommand{\qle}{{\mathscr E}}% quasilinear equation
\newcommand{\chm}{\kappa}% characteristic momentum
\newcommand{\chv}{\cX}% characteristic variety
\begin{document}
\title{Hydrodynamic integrability via geometry}
\author{David M.J. Calderbank}
\address{Department of Mathematical Sciences, University of Bath,
Claverton Down, Bath, BA2 7AY, UK.}
\email{D.M.J.Calderbank@bath.ac.uk}
\begin{abstract} This paper develops a geometric approach to the theory of
  integrability by hydrodynamic reductions to establish an equivalence, for a
  large class of quasilinear systems, between hydrodynamic integrability and
  the existence of nets compatible with the geometry induced on the codomain
  of the system. This unifies and extends known results for three subclasses
  of such systems. The generalization is obtained by studying the algebraic
  geometry of the characteristic correspondence of the system, and by
  introducing a generalized notion of conjugate nets.
\end{abstract}
%\thanks{I thank the Czech Grant Agency, grant P201/12/G028, for financial support.}
\maketitle

\section{Introduction}

The method of hydrodynamic reductions is a test for the integrability of
dispersionless systems introduced by E.~Ferapontov and
K.~Khusnutdinova~\cite{FKh3,FKhn}. It applies to PDE systems which can be
written in the following first order quasilinear form:
\begin{equation}\label{eq:qls}
A_1(\uk) \del_{t_1}\uk +\cdots+ A_n(\uk) \del_{t_n} \uk =0,
\end{equation}
where the unknown is a function $\uk\colon \as\to \cod$, $\bt=(t_1,\ldots
t_n)$ are standard coordinates on the open domain $\as\sub\t\cong\R^n$ (with
$n\geq 3$), the codomain $\cod$ is (for simplicity here) an open subset of
$\R^m$, and each $A_j\colon \cod\to M_{k\times m}(\R)$ is a given
matrix-valued function (so $k$ is the number of equations in the system).
Note that $A_1,\ldots A_n$ do not depend explicitly on $\bt$, so the system is
invariant under local translations in $\as$.

Fundamental examples include $N$-component hydrodynamic systems of the form
\begin{equation}\label{eq:hs}
\del_{t_j} R^a = \chm_{aj}(\bR) \del_{t_1} R^a, \qquad\qquad\qquad
  j\in\{2,\ldots n\},\quad a\in\{1,\ldots N\},
\end{equation}
for $\bR=(R^1,\ldots R^N)\colon \as\to \cov$, with $\cov$ open in $\R^N$, where
$\chm_{aj}$ are given functions of $\br=(r_1,\ldots r_N)\in \cov$ which satisfy
the compatibility conditions that
\[
\text{for all}\quad a\neq b\in \{1,\ldots N\}\quad\text{and}\quad
j\in\{2,\ldots n\}, \quad \del_{b}\chm_{aj} =\gam_{ab}(\br)
\,(\chm_{bj}-\chm_{aj}),
\]
with $\gam_{ab}(\br)$ independent of $j$ and $\del_b:=\del_{r_b}$.
Integrability for such systems was shown by Tsarev~\cite{Tsarev}, who called
them semi-hamiltonian systems of hydrodynamic type.

An $N$-component \emph{hydrodynamic reduction} of~\eqref{eq:qls} consists of
functions $\Uk\colon\cov\to\cod$ and $\chm_{aj}\colon\cov\to\R$, for
$a\in\{1,\ldots N\}$ and $j\in\{2,\ldots n\}$, such that~\eqref{eq:hs} is
compatible and $\uk=\Uk\circ \bR$ satisfies~\eqref{eq:qls} if $\bR=(R^1,\ldots
R^N)$ satisfies~\eqref{eq:hs}. These requirements impose a PDE system on the
functions $\Uk$ and $\chm_{aj}$, and~\eqref{eq:qls} is \emph{hydrodynamically
  integrable} if this PDE system is compatible for all $N\geq 2$ (although it
suffices to check $N=3$~\cite{FKh3,FKhn}).

Hydrodynamic integrability is known to be equivalent to the existence of a
dispersionless Lax pair in some cases~\cite{BFT,DFKN1,DFKN2,FHK,FKhc2}, but it
does not require one to find the Lax pair. Thus it has the benefit that it is
an algorithmic test of integrability (for this class of systems). However, it
is computationally intensive, and for all but the simplest cases, carrying out
the test requires symbolic computer algebra. Other algorithmic tests are
available when $n=3$ or $4$ using Einstein--Weyl or self-dual conformal
geometry respectively (called ``integrable background geometries''
in~\cite{Cal:ibg}), but these only apply when the characteristic variety of
the system (see below) is a quadric---see~\cite{BFKN,CK} for the general
relation to Lax pairs in this setting.  In some overlapping cases,
Einstein--Weyl/self-dual integrability and hydrodynamic integrability are
known to be equivalent~\cite{DFKN1,DFKN2,FK1}, but this is established by
direct computation, and a conceptual explanation is lacking.

\begin{example}\label{ex:dKP} To illustrate the process in a
case amenable to hand computation, let $n=3$, $(t_1,t_2,t_3)=(t,x,y)$ and
consider an equation of the form
\begin{equation}\label{eq:g-dKP}
(u_x + \tau(u) u_t)_t = u_{yy},
\end{equation}
for a scalar function $u$, where $\tau$ is a given function of one variable,
and $t,x,y$ subscripts denote partial derivatives. Then~\eqref{eq:g-dKP} is
equivalent to the first order quasilinear system
\begin{align}\label{eq:dKP1}
u_y-v_t&=0, & u_x + \tau(u) u_t - v_y&= 0,
\end{align}
for $\uk=(u,v)$. A hydrodynamic reduction in this case consists of functions
$U$, $V$, $\mu_a=\chm_{a2}$ and $\lam_a=\chm_{a3}$ on an open subset $\cov$ of
$\R^N$, and the compatibility of~\eqref{eq:hs} requires that
\begin{equation}\label{eq:3dcomp}
\frac{\del_b\mu_a}{\mu_b-\mu_a}=\frac{\del_b\lam_a}{\lam_b-\lam_a}
\qquad\text{for all} \qquad a\neq b.
\end{equation}
Using the chain rule for $u=U\circ \bR$ and $v=V\circ \bR$, the system
\begin{align}\label{eq:3dhs}
R^a_x &=\mu_a(\bR) R^a_t,  & R^a_y&=\lam_a(\bR) R^a_t,
\end{align}
yields
\begin{align*}
u_y-v_t&=\mathop{\textstyle\sum}_a (\lam_a\, \del_aU - \del_a V)R^a_t, &
u_x + \tau(u) u_t - v_y
&= \mathop{\textstyle\sum}_a\bigl((\mu_a+\tau(U))\,\del_a U-\lam_a\,\del_aV\bigr)
R^a_t.
\end{align*}
Hence to satisfy~\eqref{eq:dKP1} for any solution $\bR$ of~\eqref{eq:3dhs}, the
ansatz requires that
\[
\text{for all}\quad a\in\{1,\ldots N\},\quad \lam_a\, \del_aU=\del_a V\quad
\text{and}\quad (\mu_a+\tau(U))\,\del_a U=\lam_a\, \del_aV.
\]
In particular
\[
\mu_a+\tau(U)=\lam_a^2,
\]
which is called the \emph{dispersion relation} and can be used to eliminate
$\mu_a$.  Then~\eqref{eq:3dcomp} becomes
\[
\frac{2\lam_a\del_b\lam_a-\tau'(U)\del_b U}{\lam_b^2-\lam_a^2}
=\frac{\del_b\lam_a}{\lam_b-\lam_a},\quad \text{i.e.,}\quad
\del_b\lam_a=-\frac{\tau'(U)\,\del_b U}{\lam_b-\lam_a},
\]
for $a\neq b$, while the symmetry of $\del_b\del_a V=\del_b(\lam_a\,\del_a U)$
in $a,b$ implies
\[
\del_b\del_a U
=-2\frac{\tau'(U)\,\del_aU\,\del_bU}{(\lam_b-\lam_a)^2}
\]
(and then $V$ may also be eliminated). Hydrodynamic integrability requires
that no further equations for $U,\lam_a$ arise in this way. In particular, the
symmetry of $\del_c\del_b\lam_a$ in $b,c$ (for $a,b,c$ distinct) implies that
$\tau''=0$, and the symmetry of $\del_c\del_b\del_a U$ is then consistent with
the system.  Thus~\eqref{eq:g-dKP} is hydrodynamically integrable if and only
if either $\tau$ is constant, so that~\eqref{eq:g-dKP} is linear, or $\tau$
has degree $1$, so that~\eqref{eq:g-dKP} is equivalent by an affine
transformation of $u$ to the dispersionless Kadomtsev--Petviashvilli (dKP)
equation $(u_x + u u_t)_t = u_{yy}$.
\end{example}

The computational intensiveness motivates the search for a geometric
underpinning to hydrodynamic integrability, and there is much that is already
understood. First the dispersion relation has a geometric interpretation.  In
the case of equation~\eqref{eq:g-dKP}, it states that for each $a$,
$[1,\mu_a(r),\lam_a(r)]$ is a point on the ($u$-dependent) projective variety
\[
\{[\xi_1,\xi_2,\xi_3]\in \Proj(\R^3):\xi_1 \xi_2 + \tau(u) \xi_1^2=\xi_3^2\},
\]
at $u=U(r)$. This is the \emph{characteristic variety} of~\eqref{eq:g-dKP},
and for hydrodynamic reductions of general quasilinear systems~\eqref{eq:qls},
the \emph{characteristic momenta} $\theta_a=\sum_{j=1}^n\chm_{aj}\,\d t_j$,
with $\chm_{a1}=1$, always define points $[\theta_a]$ on the characteristic
variety of the equation.

In the important case $n=3$, A. Odeskii and V. Sokolov~\cite{Ode,OS} have
axiomatised and studied the systems of Gibbons--Tsarev type~\cite{GT} arising
from hydrodynamic reductions. Here the characteristic variety is a bundle of
curves over $\cod$ and so the ``Gibbons--Tsarev structure'' may be described
using $1$-parameter families of vector fields on $\cod$.

For arbitrary $n$, there are now several works~\cite{BFT,DFKN1,FHK} showing
that for particular classes of quasilinear systems~\eqref{eq:qls},
hydrodynamic reductions correspond to nice submanifolds with respect to some
interesting geometric structure on the codomain $\cod$ of $\uk$.

The main result of this paper places these latter classes in common
framework. One of the central ingredients is inspired by work of
A. Smith~\cite{Smith,Smith1,Smith-talk,Smith-notes}, who emphasises the role
not only of the characteristic variety, but also the so-called rank one
variety.  For a quasilinear system~\eqref{eq:qls}, the latter is quite
straightforward to define: the system imposes a linear constraint on the
derivative $\d \uk_z$ at each $z\in\as$, depending only on $p=\uk(z)$, hence
defining a $p$-dependent linear subspace $\qle$ of $m\times n$ matrices
($p\in\cod$); the \emph{rank one variety} is the projective variety $\cR^\qle$
in $\Proj(\qle)$ of rank one matrices in $\qle$ up to scale.  Thus it is a
$p$-dependent projective variety lying in the \emph{Segre variety} of all
$m\times n$ matrices of rank one, up to scale. The latter is canonically
isomorphic to the product $\Proj(\R^n)\times\Proj(\R^m)$ of the projective
spaces of $\t\cong \R^n$ and $\R^m$, because any rank one matrix factorizes as
a product $\theta\tens S$ of a column vector and a row vector, each determined
uniquely up to scale.

Thus the rank one variety has projections onto both $\Proj(\R^n)$ and
$\Proj(\R^m)$; the former is the characteristic variety $\chv^\qle$ of the
system.  I am unaware of a standard name for the latter projection $\cC^\qle$,
so I refer to it as the \emph{cocharacteristic variety} of the system. Note
that, geometrically speaking, the $\R^m$ appearing here is the tangent space
of the codomain $\cod$ at a given point $p\in\cod$. In this more invariant
language, the $p$-dependent linear subspace $\qle$ defining the equation is
really a vector subbundle of the vector bundle $\t\tens T\cod$ with fibre
$\t\tens T_p\cod$ at $p\in\cod$, and the cocharacteristic variety $\cC^\qle$
is a subbundle of $\Proj(T\cod)$, i.e., it defines a family of cones in the
tangent spaces of $\cod$, equipping $\cod$ with a potentially interesting
geometric structure.

A key observation about the method of hydrodynamic reductions is that the
dispersion relation arises because the characteristic momenta are projections
of points on the rank one variety. Indeed if $\uk=\Uk\circ \bR$ and $\bR$
satisfies~\eqref{eq:hs} then, by the chain rule,
\begin{align*}
\d \uk &= \bR^* \d \Uk\circ \d \bR = \sum_{a=1}^N \d R^a \tens \del_a \Uk(\bR)
=\sum_{a=1}^N f_a(\bR)\theta_a(\bR) \tens\del_a \Uk(\bR)
\end{align*}
where $\theta_a(\bR)=\sum_{j=1}^n \chm_{aj}(\bR) \d t_j$, and (assuming
$\chm_{a1}=1$ as before) $f_a(\bR)=\partial_{t_1} R^a$. For this to yield
solutions of~\eqref{eq:qls} for given $\Uk$ and any such $\bR$, $\theta_a\tens
\del_a \Uk$ must must be a section of $\qle$ for each $a\in \{1,\ldots N\}$,
and this is clearly also sufficient. Since $\theta_a\tens \del_a \Uk$ is a
rank one matrix, its span $[\theta_a\tens \del_a \Uk]$ must lie in the rank
one variety $\cR^\qle$; thus $[\theta_a]$ is characteristic and $[\del_a \Uk]$
is cocharacteristic. The latter condition means that the derivative of $\Uk$
maps coordinate lines in $\cov\sub\R^N$ onto curves in $\cod$ whose tangent
lines belong to the cocharacteristic variety $\cC^\qle$ at each point. This is
a special kind of \emph{net} (see e.g.~\cite{AG}) with respect to the
geometric structure $\cC^\qle$ defines on $\cod$.

In the classes of examples studied in~\cite{BFT,DFKN1,FHK}, such a net turns
out to be sufficient to recover the hydrodynamic reduction, as the coordinate
derivatives $\del_a \Uk$ implicitly contain information about the
characteristic momenta. This can be understood using the correspondence
between characteristic and cocharacteristic varieties given by the rank one
variety.  For well-behaved systems, this correspondence is a bijection, i.e.,
for each $[\theta]$ in the characteristic variety, there is a unique $[S]$ in
the cocharacteristic variety such that $[\theta\tens S]$ is in the rank one
variety, and vice versa. This is the case in Example~\ref{ex:dKP}: for given
$U$, the cocharacteristic variety consists of all $[u,v]$ in $\Proj(\R^2)$ and
the characteristic variety is the Veronese embedding of $\Proj(\R^2)$ as a
conic in $\Proj(\R^3)$. In general such a bijective rank one correspondence
means that the characteristic and cocharacteristic varieties are simply
different projective embeddings of the same underlying abstract variety.

Unfortunately this is not sufficient in general to describe hydrodynamic
reductions purely in terms of the cocharacteristic variety, because it is
necessary also to encode the compatibility condition of~\eqref{eq:hs}, and
this depends on the embedding of $\chv^\qle$. Motivated again
by~\cite{BFT,DFKN1,FHK}, I describe a class of \emph{compliant} quasilinear
systems for which the characteristic embedding can be recovered from the
cocharacteristic one, and a class of \emph{cocharacteristic nets} for which
the following result holds.

\begin{thm}\label{t:main}
  Let $\qle\leq\t^*\tens T\cod$ be a compliant quasilinear system and $N\leq
  \dim\cod$. Then up to natural equivalences, there is a bijection between
  generic $N$-component hydrodynamic reductions of $\qle$ and generic
  $N$-dimensional cocharacteristic nets in $\cod$.
\end{thm}
It follows that hydrodynamic integrability for compliant systems is equivalent
to the existence sufficiently many cocharacteristic nets in $\cod$, a
condition which is amenable to analysis by differential geometric and/or
representation theoretic techniques (see e.g.~\cite{BFT,FK2,Smith,Smith1}).

Note that Example~\ref{ex:dKP} is not compliant. However, even when a
quasilinear system is not compliant, there may be a reformulation of the
system (for example by prolongation or differential covering) which is, and
this can be done for the dKP equation (see Remark~\ref{r:dKP}).

The structure of the paper is as follows. In Section~\ref{sec:affine}, I
explain in more detail the affine geometry of first order quasilinear systems
and the characteristic correspondence. Section~\ref{sec:compliant} reviews the
algebraic geometry of projective embeddings in order to define the notion of a
compliant quasilinear system. In Section~\ref{sec:cc-net}, I turn to the
differential geometry of nets, and explain what is a cocharacteristic
net. This turns out to include conjugate nets as a special case. The proof of
Theorem~\ref{t:main} is then given in Section~\ref{sec:hi}.

Throughout the paper, all manifolds are assumed connected, and can be real or
complex, and all maps between manifolds are smooth over the base field (hence
holomorphic in the complex case). For a vector bundle $\mathscr E$ over a
manifold $\mathscr M$, the space of (smooth) $k$-forms on $\mathscr M$ with
values in $\mathscr E$ will be denoted $\Om^k_{\mathscr M}(\mathscr E)$, or
$\Om^k_{\mathscr M}(V)$ if $\mathscr E=\mathscr M\times V$ is a trivial
bundle; when $k=0$ this is just the space of (smooth) sections of $\mathscr E$
or functions $\mathscr M\to V$. The derivative of a map $F\colon
\mathscr M\to \mathscr N$ between manifolds is denoted
$\d F\in\Omega^1_{\mathscr M}(F^*T\mathscr N)$, but in the case
that $\mathscr N$ is a vector space $V$, $F^*TV\cong\mathscr M\times V$
is trivial, so $\d F\in\Omega^1_{\mathscr M}(V)$.

\section{Quasilinear systems and their characteristic geometry}
\label{sec:affine}

\subsection{Affine geometry of quasilinear systems}

First order quasilinear systems~\eqref{eq:qls} are translation invariant, and
so the natural context for their study is affine geometry. An open subset
$\as$ of an affine space is an example of a flat affine manifold.  More
precisely, if the affine space is modelled on an $n$-dimensional vector space
$\t$, then there is an exact $1$-form $\d\bt\in \Om^1_\as(\t)$ which defines a
bundle isomorphism $T\as\cong \as\times\t$. Under this isomorphism, the
constant $\t$-valued functions are identified with the infinitesimal
generators of the local free and transitive translation action on $\as$.  The
affine coordinate $\bt\colon \as\to \t$ is defined up to an additive constant
in $\t$, which is usually fixed by specifying an origin $z\in\as$ with
$\bt(z)=0$. I will not do this herein, nor will I assume $\as$ is an open
subset of an affine space: it suffices to have an exact $1$-form $\d\bt$, the
\emph{tautological $1$-form}, inducing $T\as\cong \as\times\t$.

\begin{defn}\label{d:aff} Let $\uk\colon\as\to\cod$ be a map between manifolds,
with domain $\as$ a flat affine $n$-manifold as above, and derivative
$\d\uk\in \Om^1_\as(\uk^*T\cod)$.  Then the \emph{affine derivative} of $\uk$
is the unique $\psi\in\Om^0_\as(\t^*\tens\uk^*T\cod)$ with
$\d\uk=\ip{\psi,\d\bt}$, where angle brackets denote contraction of $\t$ with
$\t^*$.

In particular if $f$ is a scalar-valued function on $\as$, then its affine
derivative is in $\Om^0_\as(\t^*)$; if this function is constant, $f$ is said
to be \emph{affine}, and $\h$ will denote the space of affine functions on
$\as$.
\end{defn}

\begin{rems}\label{r:aff} In view of the definition, it is natural to denote
the affine derivative by $\d\uk/\d\bt$; indeed if $\as$ is $1$-dimensional,
this is the ordinary derivative of a function of $1$-variable, expressed as a
ratio of $1$-forms. More generally, in local affine coordinates,
\[
\d\bt = (\d t_1,\ldots \d t_n) \quad\text{and}\quad \frac{\d\uk}{\d \bt} =
(\partial_{t_1} \uk,\ldots \partial_{t_n} \uk).
\]
(the partial derivatives of $\uk$). 

The affine derivative restricts to a linear map from $\h$ to $\t^*$, with
kernel the constant functions; if this map surjects (as it does locally) then
there is a natural affine coordinate $\bt\colon \as\to \h^*$ defined by
$\ip{\bt(z),f}=f(z)$. This identifies $\as$ locally with the affine hyperplane
$\ip{\bt,1}=1$ in $\h^*$; in particular, $\ip{\d\bt,c}=0$ for any constant
function $c$, so $\d\bt$ takes values in $\t$, and agrees with the
tautological $1$-form, justifying the notation.
\end{rems}
Although the above invariant language is not essential, it provides a
convenient and computationally concise way to formalize the theory of
quasilinear systems.

\begin{defn} A \emph{quasilinear system \textup(QLS\textup)}, on maps from a
flat affine space $\as$ with tautological $1$-form $\d\bt\in \Om^1_\as(\t)$ to
an $m$-manifold $\cod$, is vector subbundle $\qle\leq \t^*\tens T\cod$ over
$\cod$; a \emph{solution} of the QLS $\qle$ is a map $\uk\colon \as\to \cod$
with $\d \uk/\d\bt \in\Om^0_\as(\uk^*\qle)$.
\end{defn}

If $\qle_p \leq \t^*\tens T_p\cod$ has codimension $k$ for each $p\in \cod$,
then it is locally the kernel of $k$ independent linear forms on $\t^*\tens
T_p\cod$ depending smoothly on $p$, and so the QLS imposes $k$ linear
constraints on the derivative of $\uk$. Identifying $\t$ with $\R^n$ or $\C^n$
(via affine coordinates $t_i$) and $T\cod$ locally with $\cod\times\R^m$ or
$\cod\times\C^m$, the $t_i$ components of these linear forms combine, for each
$j\in \{1,\ldots n\}$, into $k\times m$ matrix valued functions $A_j(p)$ of
$p\in\cod$, yielding the coordinate expression~\eqref{eq:qls} for the QLS.

\begin{example} Consider the first order formulation of the generalized dKP
equation~\eqref{eq:g-dKP} for $\uk=(u,v)\colon \as=\R^3\to \cod=\R^2$. Thus
the fibre of $\qle$ at $(u,v)\in\cod=\R^2$ is the set of $(u_t,u_x,u_y)\tens
(1,0)+(v_t,v_x,v_y)\tens (0,1)$ satisfying~\eqref{eq:dKP1}, which may be
solved for $(v_t,v_y)$ to give
\[
\qle_{(u,v)}=\{(u_t,u_x,u_y)\tens (1,0)+(u_y,v_x,u_x+\tau(u)u_t)\tens (0,1)\}.
\]
\end{example}

\begin{example}\label{ex:hs} A QLS in $n=2$ dimensions is often called a
\emph{hydrodynamic system}; it takes the form $B(\uk) \del_t\uk =
C(\uk)\del_x\uk$ for affine coordinates $(x,t)$ and matrix valued functions
$B,C$.  Multiplying by $\d x\wedge \d t$, this system can be rewritten as
$(B(u)\d x +C(u)\d t)\wedge \d \uk = 0$.

This has an analogue in any dimension $n$. For later use, the unknown map will
be denoted $\bR\colon \as\to \cov$, where the codomain $\cov$ is an
$N$-manifold. The hydrodynamic system is then determined by a section $A\in
\Om^0_\cov(\t^*\tens\mathrm{End}(T\cov))$, and $\bR$ solves this system if
\begin{equation}\label{eq:hs2}
\bR^*A \wedge \frac{\d \bR}{\d\bt} = 0 \quad \text{in}\quad
\Om^0_\as(\Wedge^2 \t^*\tens \bR^*T\cov),
\end{equation}
where the wedge product is on $\t^*$ and the action of $\End(T\cov)$ on
$T\cov$ is used to multiply coefficients. Thus $\qle$ is the kernel of the
bundle map $\t^*\tens T\cov\to \Wedge^2 \t^*\tens T \cov$ sending $\psi$ to
$A\wedge\psi$. The QLS can also be formulated as an exterior differential
system (EDS) on $\as\times\cov$ using the projections $\pi_\as$ and $\pi_\cov$
to $\as$ and $\cov$: for any $1$-form $\alpha$ on $\cov$,
$\pi_\cov^*\alpha(A)\in \Omega^1_{\as\times\cov}(\t^*)$ and
$\pi_\as^*\d\bt\in\Omega^1_{\as\times\cov}(\t)$, hence their contracted wedge
product $\ip{\pi_\cov^*\alpha(A)\wedge\pi_\as^*\d\bt}$ is a $2$-form on
$\as\times\cov$. Then $\bR$ solves~\eqref{eq:hs2} if and only if these
$2$-forms pullback to zero by $(\iden_\as,\bR)\colon \as\to \as\times\cov$,
i.e., the graph of $\bR$ is an integral manifold of the EDS (differential
ideal) generated by these $2$-forms.

Evidently if $\bR=\Psi\circ\tilde\bR$ for some diffeomorphism $\Psi$ of
$\cov$, then $\d \bR = \tilde\bR^*\d \Psi\circ \d \tilde\bR$, and so the
system transforms to $\ip{\tilde\bR^*\tilde A,\d\bt}\wedge\d \tilde\bR =0$
where $\tilde A = (\d \Psi)^{-1}\circ \Psi^*A \circ \d \Psi$.

The fundamental case for hydrodynamic integrability is when $A$ can be
simultaneously diagonalized by such a diffeomorphism $\Psi$, with the
eigenvectors tangent to coordinate lines on $\cov$. Thus there are coordinates
$r_a:a\in\cA=\{1,\ldots N\}$ on $\cov$ called \emph{Riemann invariants}, and
functions $\chm_a\colon\cov\to \t^*$ called \emph{characteristic momenta} such
that $A= \sum_{a\in \cA} \chm_a \d r_a\tens \del_{r_a}$. These coordinates may
be assumed (for local questions) to realise $\cov$ as an open subset of
$\R^N$.  Setting $\theta_a=\ip{\chm_a,\d\bt}$ for $a\in\cA$, the above EDS is
generated by the $2$-forms $\theta_a\wedge\d r_a:a\in\cA$ on $\as\times\cov$,
and the QLS $\qle$ is spanned by $\chm_a\tens\del_{r_a}:a\in\cA$.

Note that there is some residual gauge freedom in the system: each $\chm_a$
may be scaled by a function $f_a$ on $\cov$ (\emph{scaling equivalence}), and
each $r_a$ may be replaced by $r_a'=\rho_a(r_a)$ for functions $\rho_a$ of one
variable (\emph{reparametrization equivalence}); these do not change $\qle$,
or the differential ideal generated by $\theta_a\wedge\d r_a:a\in\cA$.
\end{example}

I now turn to the examples that this paper seeks to unify.  These examples
have in common that they all come from a PDE system, for a function
$\bw\colon\as\to\cow$ with values in an affine space $\cow$, which is
\emph{not a priori} a QLS, but which may be reformulated as a QLS on the
derivative of $\bw$. There are three cases: general grassmannian
equations~\cite{DFKN1,DFKN2}, in which the derivative of $\bw$ satisfies an
algebraic constraint; hessian or Hirota type
equations~\cite{FHK,Smith,Smith1}, in which $\bw$ is the derivative of a
scalar function $\varphi$ whose second derivative satisfies an algebraic
constraint, and invariant wave equations~\cite{BFT}, in which the equation on
$\bw$ is linear in the second derivative with coefficients depending only on
the first derivative (thus the equation is invariant under translation of
$\bw$). These three classes of examples will be referred to as types G, H and
I respectively.
 
\begin{examples}\label{ex:main} Let $\cow$ be an affine space modelled on a
vector space $V$ with $\dim V=q$. Then the derivative of a smooth map
$\bw\colon \as\to \cow$ may be viewed as a $V$-valued $1$-form
$\d\bw\in\Omega^1_{\as}(V)$ (strictly speaking this is $\bw^*\d\bs\circ \d\bw$
where $\d\bs\colon T\cow\to V$ is the tautological $1$-form of $\cow$)---hence
it has an affine derivative $\uk=\d\bw/\d\bt\colon \as\to\t^*\tens V$

Alternatively, for each $z\in \as$, the graph of the linear map
$\uk(z)\in\t^*\tens V\cong \Hom(\t.V)$ is in $\Gr_n(\t\oplus V)$, the
grassmannian of $n$-dimensional subspaces $\t\oplus V$, so $\uk$ may be viewed
as a map $\Gamma_{\uk}\colon \as\to \Gr_n(\t\oplus V);z\mapsto (\d\bt_z, \d
\bw_z)(T_z \as)$.  This viewpoint reveals $\PGL(\t\oplus V)$ as a natural
symmetry group~\cite{DFKN1,DFKN2}, with $\t^*\tens V$ being identified as the
open subset of $n$-dimensional subspaces of $\t\oplus V$ on which the
projection to $\t$ is an isomorphism.

Now an arbitrary $\uk\colon\as\to\t^*\tens V$ arises locally in this way from
$\bw\colon \as\to \cow$ if and only if $\d\uk/\d\bt\colon\as\to
\t^*\tens\t^*\tens V$ takes values in $S^2\t^*\tens V$. With respect to affine
coordinates $\bt=(t_1,\ldots t_n)$ on $\as$, $\uk$ has components
$\uk_i\colon\as\to V$, and this is simply the statement that for all
$i,j\in\{1,\ldots n\}$, $\del_{t_i}\uk_j=\del_{t_j}\uk_i$, i.e., the
integrability condition to write $\uk_i=\del_{t_i} \bw\colon\as\to V$ for each
$i\in\{1,\ldots n\}$.

There are now two ways to impose equations on $\uk$. First $\Gamma_\uk$ can be
required to take values in an $m$-dimensional submanifold of $\Gr_n(\t\oplus
V)$.  Let $\cod$ be the corresponding submanifold of the open subset
$\t^*\tens V$. The integrability condition on $\uk$ may be conveniently
described using the map $(\iden_\as,\uk)\colon \as\to \as\times \cod$ as
follows: let $\bp\in\Omega^0_{\as\times\cod} (\t^*\tens V)$ be the
tautological coordinate given by projection onto $\cod\sub \t^*\tens V$ and
let $\ip{\d\bp\wedge\d\bt}\in\Omega^2_{\as\times\cod}(V)$ be the wedge product
of $\d\bp$ with $\pi_\as^*\d\bt\in \Omega^1_{\as\times\cod}(\t)$ where their
values are contracted using the natural map $\t^*\tens V\times \t\to V$.  Then
the symmetry of $\d\uk/\d\bt$ means equivalently that $(\iden_\as,\uk)^*
\ip{\d\bp\wedge\d\bt}=0$. This motivates a second class of quasilinear
constraints on $\uk$ that $(\iden_\as,\uk)^*\ip{\d\bp\wedge\Phi}=0$ for
further differential forms $\Phi\in\Omega^k_{\as\times\cod}(\t)$ of the form
$\Phi=\sum_{I,j} F_{I,j}(\bp) \d t_I\tens \del_{t_j}$, where $I=\{i_1,\ldots
i_k\}$ is a multi-index with $1\leq i_1<i_2<\cdots< i_k\leq n$ and $\d t_I =\d
t_{i_1}\wedge \d t_{i_2}\wedge\cdots \wedge\d t_{i_k}$. Thus the QLS is the
EDS generated by components of $\ip{\d\bp\wedge\d\bt}$ and the additional
differential forms $\ip{\d\bp\wedge\Phi}$.

The three classes G,H,I of QLS are special cases of these EDS.
\begin{numlist}
\item[(G)] This is the case of a general submanifold $\cod$, with no
  additional differential forms $\Phi$. Since $\cod$ has dimension $m$, it has
  codimension $m'=qn-m$ and so may be described locally by equations
  $F_k(\bp)=0$ for $k\in\{1,\ldots m'\}$. If $\uk\colon \as\to\cod$ with
  $(\iden_\as,\uk)^*\ip{\d\bp\wedge\d\bt}=0$ then $\uk$ comes from a smooth
  map $\bw\colon\as\to\cow$ which solves the system
\begin{equation*}
  F_k(\d\bw/\d\bt) = 0\qquad\text{for all} \qquad k\in\{1,\ldots m'\}.
\end{equation*}
For any $p\in\cod$, the fibre $\qle_p$ of the QLS is the intersection of
$\t^*\tens T_p\cod$ with $S^2\t^*\tens V$.
\item[(H)] This is an important special case of (G) where $\cow=V=\t^*$ and
  the submanifold is a hypersurface in the lagrangian grassmannian of maximal
  isotropic subspaces in $\t\oplus V=\t\oplus \t^*$ with respect to the
  natural symplectic form given by $\omega(X_1+\xi_1,X_2+\xi_2)=
  \ip{\xi_1,X_2}-\ip{\xi_2,X_1}$. The lagrangian grassmannian has $S^2\t^*$ as
  an open subset which, given affine coordinates $\bt=(t_1,\ldots t_n)$, has
  coordinates $p_{ij}=p_{ji}$ ($i,j\in \{1,\ldots n\}$). Hence the
  hypersurface $\cod$ in $S^2\t^*$ is given locally by $F(\ldots,
  p_{ij},\ldots)=0$ for some function $F$ of $\frac12 n(n+1)$ variables.  Any
  $\uk\colon \as\to\cod$ with $(\iden_\as,\uk)^*\ip{\d\bp\wedge\d\bt}=0$ now
  comes from a scalar function $\varphi$ on $\as$ (via $\bw=\d\varphi/\d\bt$)
  which solves the equation
\begin{equation*}
F(\ldots,\del^2_{t_i,t_j} \varphi,\ldots) = 0.
\end{equation*}
In this case, for any $p\in\cod$, $T_p\cod\leq S^2\t^*$ is the kernel of $\d
F_p$, so $\qle_p$ is the intersection of $\t^*\tens T_p\cod$ with $S^3\t^*$.

\item[(I)] Let $q=1$ (so that $V$ is one-dimensional), with $\cod$ open in
  $\t^*\tens V\sub\Gr_n(\t\oplus V)$, and suppose
  $Q(\bp)\in\Omega^0_{\as\times\cod}(S^2\t)$ is a quadratic form on $\t^*$
  with coefficients depending only on $\bp$. Thus in affine coordinates
  $\bt=(t_1,\ldots t_n)$ with $\bp=(p_1,\ldots p_n)$, $Q=\sum_{i,j}
  F_{ij}(p_1,\ldots p_n)\del_{t_i}\tens \del_{t_j}$ with $F_{ij}=F_{ji}$, and
  defines $\Phi\in \Omega^{n-1}_{\as\times\cod}(\t)$ by the partial contraction
  $\Phi=\sum_{i,j}F_{ij}(p_1,\ldots p_n)\del_{t_i}\interior (\d
  t_1\wedge\cdots\wedge \d t_n)\tens \del_{t_j}$. Now if $(\iden_{\as},\uk)$
  pulls back both $\ip{\d\bp\wedge\d\bt}$ and $\ip{\d\bp\wedge\Phi}$ to zero,
  then $\uk$ comes from a scalar function $w$ on $\as$ which solves the
  equation
\begin{equation*}
\sum_{i,j} F_{ij}(\del_{t_1}w,\ldots \del_{t_n}w)\del^2_{t_i,t_j} w = 0.
\end{equation*}
At each $p\in \cod$, $\qle_p$ is the kernel of $Q(p)\tens\iden_V$ in $\t^*\tens
T_p\cod\cap S^2\t^*\tens V$.
\end{numlist}
\end{examples}

\begin{rem}\label{r:dKP}  Up to differential coverings, the dKP equation can
be reformulated as a QLS in each of these ways~\cite{BFT,DFKN1,FHK}. For
example, the equation
\begin{equation}\label{eq:pdKP}
  w_{xt} + w_t w_{tt} = w_{yy}
\end{equation}
has type I and so defines a QLS on the first derivatives $f=w_t$, $g=w_x$,
and $h=w_y$:
\[
f_x= g_t, \quad f_y=h_t,\quad g_y=h_x \quad\text{and}\quad f_x - f f_t = h_y.
\]
On the other hand, by differentiating~\eqref{eq:pdKP} with respect to $t$, it
follows that $u=w_t$ satisfies the dKP equation $(u_x+uu_t)_t=u_{yy}$. Finally,
writing~\eqref{eq:pdKP} as $(w_x+\frac12 w_t^{\,2})_t = w_{yy}$ yields a first
order form
\[
w_x+\tfrac12 w_t^{\,2} = v_y, \quad v_t = w_y
\]
of type G, and a further potential $\varphi$ with $\varphi_y=v$ and
$\varphi_t=w$ reduces this system to the equation $\varphi_{xt}+\tfrac12
\varphi_{tt}^{\;2}=\varphi_{yy}$ of type H.
\end{rem}

\subsection{The characteristic correspondence}

The projective bundle $\Proj(\t^*\tens T\cod)\to\cod$ has a subbundle
$\cR$ whose fibre at $p\in\cod$ is
\[
\cR_p:=\{[\xi\tens Z]\in \Proj(\t^*\tens T_p\cod) \st \xi\in\t^*,\, Z\in
T_p\cod\}.
\]
This is the image of the \emph{Segre embedding} of
$\Proj(\t^*)\times\Proj(T_p\cod)$, consisting of the projectivizations of rank
one tensors.

\begin{defn} Let $\qle\leq \t^*\tens T\cod$ be a QLS.
\begin{bulletlist}
\item The \emph{rank one variety} of $\qle$ is $\cR^\qle:=\cR\cap\Proj(\qle)$.
\item The \emph{characteristic} and \emph{cocharacteristic varieties} of
  $\qle$ are the projections $\chv^\qle$ and $\cC^\qle$ of $\cR^\qle$ to
  $\cod\times\Proj(\t^*)$ and $\Proj(T\cod)$ respectively.
\item The \emph{characteristic correspondence} of $\qle$ is the diagram
\begin{diagram}[size=1.3em,nohug]
& & \cR^\qle \\
&\ldTo^{\pi_\chv} & &\rdTo^{\pi_\cC} \\
\llap{$\cod\times\Proj(\t^*)\supseteq{}$} \chv^\qle &&&&
\cC^\qle\rlap{${}\sub\Proj(T\cod)$}
\end{diagram}
where $\pi_\chv$ and $\pi_\cC$ are the natural projection maps.
\end{bulletlist}
\end{defn}
\noindent 
For convenience, $\cR^\qle$, $\chv^\qle$ and $\cC^\qle$ will be assumed to be
fibre bundles over $\cod$ whose fibres are projective varieties, so that the
characteristic correspondence is a double fibration.

\begin{rem}\label{r:determined} For (complex, or real hyperbolic) determined
systems, the characteristic variety $\chv^\qle$ is a hypersurface, hence of
dimension $n-2$. Thus $\Proj(\qle_p)$ does not meet generic fibres of $\cR_p$
over $\Proj(\t^*)$, and the fibres it does meet, it generically meets in
dimension zero, hence in a single point because these fibres are projective
subspaces in the Segre embedding.  For $m\geq n$, the same is true for the
fibres over $\Proj(T_p\cod)$, and hence the characteristic correspondence maps
are isomorphisms, and $\cC^\qle_p$ has codimension $(m-1)-(n-2)=m-n+1$ in
$\Proj(T_p\cod)$.
\end{rem}

\setcounter{example}{0}
\renewcommand{\theexample}{\thesection.\arabic{example} \textit{bis}}

\begin{example}
For the generalized dKP equation~\eqref{eq:g-dKP}, the fibre of the rank one
variety $\cR^\qle$ at $(u,v)$ consists of the nonzero elements with
$(u_t,u_x,u_y)$ and $(u_y,v_x,u_x+\tau(u) u_t)$ linearly dependent, which
forces
\[
u_t v_x = u_y u_x, \quad u_y v_x=(u_x+\tau(u)
u_t) u_x,\quad \text{and}\quad u_y^2=u_t(u_x+\tau(u) u_t).
\]
The last equation is solved by
\[
u_t = c\lam_0^2,\quad u_x = c(\lam_1^2-\tau(u)\lam_0^2),\quad \text{and}\quad
u_y =c\lam_0\lam_1
\]
for constants $c,\lam_0,\lam_1$. Hence $c\lam_0^2 v_x =
c^2\lam_0\lam_1(\lam_1^2-u\lam_0^2)$ and $c\lam_0\lam_1 v_x =
c^2\lam_1^2(\lam_1^2-\tau(u)\lam_0^2)$, so either $c\lam_0=c\lam_1=0$ or
$\lam_0 v_x=c\lam_1(\lam_1^2-u\lam_0^2)$. Now if $\lam_0=0$, $c\lam_1=0$,
otherwise $c$ is divisible by $\lam_0$; thus, without loss, $c=\lam_0$,
$v_x=\lam_1^2(\lam_1^2-\tau(u)\lam_0^2)$ and
\[
\cR^\qle_{(u,v)}=\{(\lam_0^2,\lam_1^2-\tau(u)\lam_0^2,\lam_0\lam_1)\tens
(\lam_0,\lam_1): \lam_0,\lam_1\in\R\},
\]
the cocharacteristic variety is $\cC^\qle=P^1$, and the characteristic variety
$\chv^\qle$ is a $u$-dependent conic in $P^2$.
\end{example}

\begin{example} For a (diagonalizable) hydrodynamic system, with Riemann
invariants and characteristic momenta $r_a,\chm_a:a\in\cA$, the description of
$\qle$ immediately yields that $[\chm_a\tens \del_{r_a}]$ are in the rank one
variety. It follows generically (if the $\chm_a$ are pairwise independent)
that
\[
\chv^\qle=\{[\chm_a]:a\in\cA\}, \quad\cC^\qle =\{[\del_{r_a}]:a\in\cA\}\quad
\text{and}\quad \cR^\qle=\{[\chm_a\tens\del_{r_a}]:a\in\cA\}.
\]
\end{example}

\begin{examples} For a QLS $\qle$ of type G, H or I, it is straightforward
to describe the characteristic correspondence using the description of $\qle$
given previously.
\begin{numlist}
\item[(G)] At each $p\in\cod$,
\begin{align*}
\chv^\qle_p &= \{[\xi]\in\Proj(\t^*)\st \xi\tens v\in T_p\cod
\text{ for some nonzero } v\in V\}\\
\cC^\qle_p &= \{[\xi\tens v]\in\Proj(T_p\cod)\}\\
\cR^\qle_p &= \{[\xi\tens\xi\tens v]\in\Proj(\t^*\tens T_p\cod)\}.
\end{align*}
Now $\Proj(\t^*)\times\Proj(V)$ has dimension $n+q-2$, so $\Proj(T_p\cod)$
generically meets it in a variety of dimension $n+q-m'-2$, so the system is
determined for $m'=q$, in which case for any $[\xi]\in\chv^\qle$ there is a
unique $[v]\in\Proj(V)$ such that $[\xi\tens v]\in \cC^\qle$. The
references~\cite{DFKN1,DFKN2} concern the cases $q=m'=2$, with $n\in\{3,4\}$.

\item[(H)] In this case $\d F_p$ induces a quadratic form $Q_p$ on $\t^*$ for
  each $p\in\cod$, and
\begin{align*}
\chv^\qle_p &= \{[\xi]\in\Proj(\t^*)\st Q_p(\xi)=0\}\\
\cC^\qle_p &= \{[\xi\tens\xi]\in\Proj(S^2\t^*)\st Q_p(\xi)=0\}\\
\cR^\qle_p &= \{[\xi\tens\xi\tens\xi]\in\Proj(S^3\t^*)\st Q_p(\xi)=0\}.
\end{align*}

\item[(I)] Again there is a quadratic form $Q_p$ (up to scale) on $\t^*$ for
each $p\in \cod$, and
\begin{align*}
\chv^\qle_p &= \{[\xi]\in\Proj(\t^*)\st Q_p(\xi)=0\}\\
\cC^\qle_p &= \{[\xi\tens v]\in\Proj(\t^*\tens V)\st Q_p(\xi)
=0\}\\
\cR^\qle_p &= \{[\xi\tens\xi\tens v]\in\Proj(S^2\t^*\tens V)\st
Q_p(\xi)=0\}.
\end{align*}
\end{numlist}
Types H and I are thus determined equations wherever $Q_p\neq0$.
\end{examples}

\renewcommand{\theexample}{\thesection.\arabic{example}}

\section{Algebraic geometry: compliant quasilinear systems}\label{sec:compliant}

\subsection{Projective embeddings} Motivated by the previous section, I focus
now on the case that the fibres over each $p\in\cod$ of the characteristic,
cocharacteristic and rank one varieties are projective embeddings of the same
abstract variety.  This is a standard situation in elementary algebraic
geometry, where projective embeddings are described using line bundles (see
e.g.~\cite{GH}).

Let $\Xi\sub \Proj(V)$ be a projective variety and let $L\to \Xi$ be the
restriction to $\Xi$ of the dual tautological line bundle $\cO_V(1)\to
\Proj(V)$, whose fibre at $\ell\in \Proj(V)$ is the dual vector space to
$\ell\leq V$: $\cO_V(1)_\ell=\ell^*$. Since the space of regular global
sections $H^0(\Proj(V),\cO_V(1))$ is $V^*$, restriction to $\Xi$ defines a
canonical linear map $V^*\to H^0(\Xi,L)$. Furthermore, this linear map is
injective unless $\Xi$ is contained in a hyperplane in $\Proj(V)$.

\begin{defn} A \emph{linear system} on an abstract variety $\Xi$ is a line
bundle $L\to\Xi$ together with a linear subspace $W$ of $H^0(\Xi,L)$, and its
\emph{base locus} is $B=\{x\in\Xi\st \forall\, w\in W,\; w(x)=0\}$. The linear
system is said to be \emph{complete} if $W=H^0(\Xi,L)$ and
\emph{basepoint-free} if $B=\varnothing$.
\end{defn}

The reason for these definitions is that if $(L,W)$ is a basepoint-free linear
system on $\Xi$, then there is a map $\Xi\to \Proj(W^*)$ sending $x\in\Xi$ to
the element of $\Proj(W^*)$ corresponding to the hyperplane $\{w\in
W:w(x)=0\}$ (the annihilator of this hyperplane). If this map is an embedding,
the linear system is said to be \emph{very ample}. In particular, if the
complete linear system yields an embedding, the line bundle $L$ is said to be
very ample.

\begin{lemma}\label{l:mult} Let $L_1$ and $L_2$ be line bundles over $\Xi$.
Then there is a canonical linear map
\[
H^0(L_1)\tens H^0(L_2) \to H^0(L_1\tens L_2); \quad \ell_1\tens\ell_2\mapsto
\ell
\]
defined by pointwise multiplication: $\ell(x)=\ell_1(x)\tens \ell_2(x)$.
\end{lemma}
\begin{proof} The map sending $(\ell_1,\ell_2)\in H^0(L_1)\times H^0(L_2)$
to $\ell_1(x)\tens\ell_2(x)$ is regular in $x$, so induces a section of
$L_1\tens L_2$; this is bilinear in $(\ell_1,\ell_2)$, and so induces a map on
the tensor product.
\end{proof}
Similarly, using the canonical isomorphism $(L_1^*\tens L_2)\tens L_1 \cong
L_2$, there is a canonical pointwise multiplication map
\[
H^0(\Xi,L_1^*\tens L_2)\tens H^0(\Xi,L_1)\to H^0(\Xi,L_2)
\]
with transpose
\begin{equation}\label{eq:pmlb}
\Phi\colon H^0(\Xi,L_2)^* \to H^0(\Xi,L_1^*\tens L_2)^*\tens H^0(\Xi,L_1)^*.
\end{equation}

\begin{prop}\label{p:main} Suppose $L_1$ and $L_2$ are very ample line bundles
on $\Xi$.  For $j\in\{1,2\}$, let $W_j=H^0(\Xi,L_j)$, $\phi_j\colon \Xi\to
\Proj(W_j^*)$ the corresponding projective embeddings, $V=H^0(\Xi,L_1^*\tens
L_2)$, and $\Phi\colon W_2^*\to W_1^*\tens V^*$ be given by~\eqref{eq:pmlb}.
Then for any $x\in \Xi$, $\Phi$ maps $\phi_2(x)$ into $\phi_1(x)\tens L(x)$
for some one dimensional subspace $L(x)$ of $V^*$.
\end{prop}
\begin{proof} Let $\ell_1\in W_1$ and $\ell\in V$; then for any $x\in \Xi$
and $\alpha\in \phi_2(x)$, $\ip{\Phi(\alpha), \ell_1\tens \ell}=
\ip{\alpha,\ell_1(x)\ell(x)}$, which vanishes whenever $\ell_1(x)=0$, i.e.,
$\ell_1\in\ker \phi_1(x)$. Let $e_1,\ldots e_{k-1}$ be a basis of
$\ker\phi_1(x)$ and extend by $e_k$ to a basis of $W_2$. Write
$\Phi(\alpha)=\sum_{j=1}^k \eps_j\tens \alpha_j$ where $\eps_1,\ldots \eps_k$
is the dual basis and $\alpha_j\in V^*$; it then follows by evaluating on
$e_j\tens v$ for $j\in \{1,\ldots k-1\}$ and any $v\in V$ that
$\Phi(\alpha)=\eps_k\tens \alpha_k$. But $\eps_k$ vanishes on $\ker\phi_1(x)$
and so $\Phi(\alpha)\in \phi_1(x)\tens\vspn{\alpha_k}$.
\end{proof}

In general, there is no reason to suppose that $\Phi$ here injects; indeed $V$
could even be zero. The injectivity can be viewed as a relative ampleness
condition.

\begin{defn}\label{d:relamp} Let $L_1$ and $L_2$ be line bundles on $\Xi$.
If the canonical map $\Phi$ in~\eqref{eq:pmlb} is injective (equivalently
$\Phi^*$ surjects), $L_2$ is said to be \emph{more ample} than $L_1$.
\end{defn}
\begin{example} The degree $d$ line bundles $\cO_V(d)$ over a projective
space $\Proj(V)$ are very ample for $d\geq 1$ and then $\cO_V(d_1)$ is more
ample than $\cO_V(d_2)$ iff $d_1\geq d_2$.
\end{example}

These ideas apply fibrewise to the bundles of projective varieties $\chv^\qle$
and $\cC^\qle$: the corresponding fibrewise dual tautological line bundles
pull back to line bundles $\Lb_\chv\to\chv^\qle$ and $\Lb_\cC\to \cC^\qle$.

For a line bundle $\Lb$ over a bundle of projective varieties $\cY\to\cod$,
let $H^0(\Lb)\to\cod$ denote the bundle of fibrewise regular sections. Then
there are canonical maps
\begin{equation}\label{eq:canmaps}
\cod\times \t\to H^0(\Lb_\chv) \quad\text{and}\quad T^*\cod\to
H^0(\Lb_\cC)
\end{equation}
given by restricting fibrewise sections of the dual tautological line bundles
over $\Uk\times\Proj(\t^*)$ and $\Proj(T\cod)$ to $\chv^\qle$ and $\cC^\qle$
respectively.

If $\chv^\qle$ and $\cC^\qle$ are not contained (fibrewise) in any hyperplane,
these maps are injective, hence fibrewise linear systems, and surjectivity
means that these linear systems are complete.  In this situation, the
projective embeddings can be recovered from the line bundles, and the relative
ampleness of Definition~\ref{d:relamp} can be used (fibrewise) to compare
them.

\subsection{Compliant quasilinear systems}

The following definition covers a wide class of QLS, yet makes the algebraic
geometry of the characteristic and cocharacteristic projective embeddings as
straightforward as possible, allowing an easy comparison between them.

\begin{defn} A QLS $\qle\leq \t^*\tens T\cod$ is \emph{compliant} if all of
the following conditions hold:
\begin{numlist}
\item the characteristic correspondence maps $\pi_\chv$ and $\pi_\cC$ are
  isomorphisms, with $\zeta^\qle= \pi_\chv\circ\pi_\cC^{-1}\colon
  \cC^\qle\to\chv^\qle$ denoting the induced isomorphism;

\item the canonical maps $\cod\times \t\to H^0(\Lb_\chv)$ and $T^*\cod\to
  H^0(\Lb_\cC)$ of~\eqref{eq:canmaps} are isomorphisms;

\item $\cV^\qle:=H^0(\Lb_\cC\tens(\zeta^\qle)^* \Lb_\chv^*)^*\to\cod$ is a
  vector bundle over $\cC^\qle$, and $\Lb_\cC$ is more ample than
  $(\zeta^\qle)^*\Lb_\chv$, so that the canonical vector bundle map
\[
  \Phi^\qle\colon T\cod\to \t^*\tens \cV^\qle,
\]
defined fibrewise as in~\eqref{eq:pmlb}, using the isomorphisms in (2), is
injective.

\item if $\mathrm{rank}(\cV^\qle)\geq 2$, no $2$-dimensional submanifold
  $\Sig$ of $\cod$ has $\Phi^\qle(T\Sig)\sub \t^*\tens \cV^\qle$
  everywhere decomposable (i.e., with all elements of rank one).
\end{numlist}
\end{defn}
Compliancy may seem rather restrictive due to the number of conditions
involved. However, by Remark~\ref{r:determined}, condition (1) is expected for
determined QLS. A significant part of condition (2) is that the characteristic
variety $\chv^\qle$ does not lie in a projective hyperplane bundle, but this
could be taken as part of what it means for a QLS to be dispersionless (I have
not found a rigorous definition of ``dispersionless'' for arbitrary QLS in the
literature).

The key condition here is (3), giving a tensor product decomposition of
$T\cod$. This fails in Example~\ref{ex:dKP} because the characteristic
embedding is more ample than the cocharacteristic one, rather than
vice-versa. However all of the QLS in Examples~\ref{ex:main} are generically
compliant. Condition (4) implies that the generic tangent space to any
$2$-dimensional submanifold $\Sig$ of $\cod$ contains indecomposable elements
(via the injection $\Phi^\qle$).  This is a technical condition needed to
ensure that the decomposition in (3) is enough to recover the compatiblity of
the hydrodynamic system from the cocharacteristic variety.  \medbreak

The following is an immediate consequence of Proposition~\ref{p:main}
(fibrewise), and allows the recovery of the characteristic embedding from the
cocharacteristic one.

\begin{prop}\label{p:comp} Let $\qle$ be a compliant QLS. Then for
any $[X]\in \cC^\qle$, there exist nonzero $\chm\in \zeta([X])$ and
$v\in\cV^\qle$ such that $\Phi^\qle(X)=\chm\tens v$.
\end{prop}

\section{Differential geometry: cocharacteristic nets}\label{sec:cc-net}

\subsection{Nets and conjugate nets}

\begin{defn} Let $\ned$ be an $N$-manifold and $\cA=\{1,\ldots N\}$.
\begin{bulletlist}
\item A \emph{pre-net} on $\ned$ is a direct sum decomposition
  $T\ned=\Dsum_{a\in\cA} \cD_a$ into rank one distributions $\cD_a\leq T\ned$ for
  $a\in\cA$; thus each $\cD_a$ is tangent to a foliation of $\ned$ with one
  dimensional leaves.
\item A pre-net $\cD_a:a\in\cA$ on $\ned$ is \emph{integrable} if for every
  subset $\cB\sub\cA$, $\cD_\cB:=\Dsum_{b\in\cB} \cD_b$ is an integrable
  distribution (i.e., tangent to a foliation with $\#\cB$ dimensional leaves);
  an integrable pre-net is called a \emph{net}.
\end{bulletlist}
\end{defn}
Recall that by the Frobenius Theorem (see e.g.~\cite{Lee}), a constant rank
distribution is tangent to a foliation if and only if its sheaf of smooth
sections is closed under Lie bracket.

\begin{prop} For a pre-net $\cD_a:a\in\cA$ on $\ned$, the following are equivalent.
\begin{numlist}
\item $\cD_a:a\in\cA$ is integrable, i.e., a net.
\item For every $\cB\sub\cA$ with $\#\cB=3$, $\cD_{\cB}$ is an integrable
  distribution, and if $N=3$, the same is true when $\#\cB=2$.
\item For every $\cB\sub\cA$ with $\#\cB=2$, $\cD_{\cB}$ is an integrable
  distribution.
\item For each $a\in\cA$, $\cD_{\cA\setminus\{a\}}$ is an integrable
  distribution.
\item Near any $p\in \ned$, there are local coordinates $r_a:a\in\cA$ such that
  $\cD_a=\mathrm{span}\{\del_{r_a}\}$.
\end{numlist}
\end{prop}
\begin{proof} It is trivial that (1) $\Rightarrow$ (2).

(2) $\Rightarrow$ (3). This is vacuous for $N\leq 2$ and built into the
  statement for $N=3$. For $N\geq 4$ and $\#\cB=2$, there exist distinct
  $b,c\in \cA\setminus\cB$, so that $[\cD_{\cB},\cD_{\cB}]\leq
  \cD_{\cB\cup\{b\}}\cap \cD_{\cB\cup\{c\}}=\cD_\cB$.

(3) $\Rightarrow$ (4). For any $b,c\in\cA\setminus\{a\}$ and any local
  sections $X_b$ and $X_c$ of $\cD_b$ and $\cD_c$ respectively, $[X_b,X_c]$ is
  a section of $\cD_b\oplus\cD_c\leq\cD_{\cA\setminus\{a\}}$, and hence
  $\cD_{\cA\setminus\{a\}}$ is integrable.

(4) $\Rightarrow$ (5). The $r_a$ are pullbacks of local coordinates on the
  one dimensional local leaf spaces of $\cD_{\cA\setminus\{a\}}$.

(5) $\Rightarrow$ (1). For $\cB\sub\cA$, $\cD_{\cB}$ is the joint kernel of
  $\d r_a:a\in \cA\setminus\cB$, hence integrable.
\end{proof}

\begin{example} Consider the hydrodynamic system $\bR^*A\wedge(\d\bR/\d\bt)=0$
where $A$ is a section of $\t^*\tens\mathrm{End}(T\cov)$ over $\cov$, as in
Example~\ref{ex:hs}.  If $A$ is everywhere simultaneously diagonalizable with
distinct eigenfunctions, then the rank one eigendistributions define a pre-net
on $\cov$. The condition that these eigendistributions are tangent to
coordinate lines (leading to the QLS $\qle$ spanned by
$\chm_a\tens\del_{r_a}:a\in\cA$) is precisely the condition that this pre-net
is a net.
\end{example}

\subsection{Conjugate and cocharacteristic nets}

For application to hydrodynamic integrability, a special class of nets will be
needed. Suppose that $\cD_a:a\in\cA$ is a pre-net on $\ned$, and that $T\ned\leq
\t^*\tens \cV$ for a vector space $\t^*$ and a line bundle $\cV\to \ned$; then
each $\cD_a$ defines a line subbundle $M_a$ of $\ned\times\t^*$ with $\cD_a=M_a
\tens\cV$.

One may then require, for all $a,b\in\cA$ and for any sections $X_b$ of
$\cD_b$ and $\sigma_a$ of $M_a$, that $\d_{X_b}\sigma_a$ is a section of
$M_a\oplus M_b\leq \ned\times\t^*$. If this holds then for all $a,b\in\cA$,
$[X_a,X_b]$ is a section of $\cD_a\oplus\cD_b$ and $\cD_a:a\in\cA$ is
therefore a net, which will be called a \emph{conjugate net}.

\begin{example} The terminology here comes from the classical situation (see
e.g.~\cite{AG}) that $\ned$ is an affine space with translations $\t^*$ (or
indeed a flat affine manifold). Then the tangent bundle of $\ned$ is
isomorphic to $\ned\times \t^*$ and a net $\cD_a:a\in\cA$ is conjugate if for
all $a,b\in\cA$, the coordinate lines along the any surface tangent to
$\cD_a\oplus\cD_b$ are conjugate, i.e., orthogonal with respect to the second
fundamental form of the surface.
\end{example}

\begin{defn}\label{d:cnet} Let $\qle\leq\t^*\tens T\cod$ be a compliant QLS
with $T\cod\leq \t^*\tens\cV^\qle$ and cocharacteristic variety $\cC^\qle\leq
\Proj(T\cod)$.  An $N$-dimensional \emph{cocharacteristic net} in $\cod$ is a
parametrized submanifold $\Uk\colon\cov\to \ned\sub \cod$, with $\cov$ open in
$\R^N$, such that:
\begin{numlist}
\item\label{i:cnet} the net on $\ned$ spanned by $\del_a
  \Uk:a\in\cA=\{1,\ldots N\}$ satisfies $[\del_a \Uk]\in\cC^\qle$; and
\item\label{i:cnet2} if $\cV^\qle$ has rank one, the net is conjugate.
\end{numlist}
\end{defn}

\section{Hydrodynamic integrability}\label{sec:hi}

\subsection{Integrable hydrodynamic systems}\label{sec:ihs}

As a warm-up for the proof of the main theorem, and to motivate the definition
of hydrodynamic integrability, consider again Example~\ref{ex:hs} in the
diagonalizable case, where the QLS $\qle$ is spanned by
$\chm_a\tens\del_{r_a}:a\in\cA$. Since the system is equivalently the EDS
generated by the $2$-forms $\theta_a\wedge\d r_a$, where
$\theta_a=\ip{\chm_a,\d \bt}$, it is compatible if and only if these $2$-forms
algebraically generate a differential ideal, i.e., for all $a\in\cA$,
$\d\theta_a\wedge\d r_a=0$ mod $(\theta_b\wedge\d r_b)_{b\in\cA}$. This holds
if and only if \[ \text{for all $a\in\cA$, $\d\chm_a\wedge\d r_a=0$ mod
  $(\chm_b\,\d r_b)_{b\in\cA}$ (as $\t^*$-valued $2$-forms on $\cov$),}
\]
i.e., for all $a\in \cA$, there are (scalar-valued) $1$-forms
$\beta_a=\sum_{b\in \cA} \beta_{ab} \d r_b$ and functions $\gamma_{ab}$ on
$\cov$ such that
\begin{equation}\label{eq:scalar}\begin{split}
\d\chm_a\wedge\d r_a&=\beta_a\wedge(\chm_a\,\d r_a)
+{\textstyle\sum_{b\in\cA}(\chm_b\,\d r_b)\wedge(\gamma_{ab}\,\d r_a)}\\
\text{i.e.,}\quad \del_b \chm_a &= \beta_{ab}\chm_a+ \gamma_{ab}\chm_b
\quad\text{for all}\quad b\neq a.
\end{split}
\end{equation}
Fixing $\beta_a$ and $\gamma_{ab}$, $\chm_a:a\in\cA$ may be viewed as a
$\t^*$-valued solution to a linear system on $\cov$. This linear system in
turn has a compatibility condition: differentiating once more,
\begin{equation*}
\begin{split}
0 &= \d\beta_a\wedge(\chm_a\,\d r_a)-\beta_a\wedge \d\chm_a\wedge\d r_a
+{\textstyle\sum_{b\in\cA}\bigl( \gamma_{ab}\,\d\chm_b+\chm_b\,\d\gamma_{ab}
\bigr)\wedge\d r_b\wedge\d r_a}\\
&=\chm_a\,\d\beta_a\wedge\d r_a+{\textstyle\sum_{b\in\cA}\chm_b\,\bigl(
\gamma_{ab}(-\beta_a+\beta_b)
+\d\gamma_{ab}\bigr)\wedge\d r_b\wedge\d r_a}\\
&\qquad+{\textstyle\sum_{b,c\in\cA}
\chm_c\,\gamma_{ab}\gamma_{bc}\,\d r_c\wedge\d r_b\wedge\d r_a}\\
&=\chm_a\,\d\beta_a\wedge\d r_a+{\textstyle\sum_{c\in\cA}\chm_c\,\bigl(
\gamma_{ac}(\beta_c-\beta_a)+\d\gamma_{ac}
-\sum_{b\in\cA}\gamma_{ab}\gamma_{bc}\d r_b\bigr)\wedge \d r_c\wedge\d r_a}
\end{split}\end{equation*}
for all $a\in \cA$. This is satisfied for all $\chm_a:a\in\cA$ if and only
if for all $a,c\in \cA$ with $c\neq a$
\begin{equation}\label{eq:hyd-int}
\begin{split}
\d\beta_a\wedge\d r_a=\d(\beta_a\wedge\d r_a) =\d(r_a\d\beta_a)&=0\\
{\textstyle\sum_{b\in\cA}\bigl(\del_b\gamma_{ac}+
\gamma_{ac}\beta_{cb}-\gamma_{ac}\beta_{ab}-\gamma_{ab}\gamma_{bc}
\bigr)\d r_b\wedge \d r_c\wedge \d r_a}&=0.
\end{split}
\end{equation}
If~\eqref{eq:hyd-int} holds, and $(\chm_a)_{a\in\cA}$ is a $\t^*$-valued
solution to~\eqref{eq:scalar} then $\theta_a\wedge\d r_a:a\in\cA$ generate a
differential ideal, and $\qle$ may be called an \emph{integrable} hydrodynamic
system.  Such systems were introduced by Tsarev~\cite{Tsarev} as
\emph{semi-hamiltonian systems of hydrodynamic type}. Such systems are
regarded as integrable because of the fundamental observation of
Tsarev~\cite{Tsarev} that they admit \emph{generalized hodograph} solutions.
Interpreted in the affine geometry of Definition~\ref{d:aff} and
Remark~\ref{r:aff} (using in particular the space of affine functions
$\h\to\t^*$ and the tautological affine coordinate $\bt\colon\as\to \h^*$),
his observation is as follows.

\begin{prop} Let $\qle$ be a hydrodynamic system with characteristic
  momenta $\chm_a:{a\in\cA}$. Then for any lift $\tilde\chm_a:{a\in\cA}$ of
  $\chm_a:{a\in\cA}$ to an $\h$-valued solution of~\eqref{eq:scalar}, any
  function $\bR\colon \as\to\cov$ solving the implicit equations
\begin{equation}\label{eq:impl}
\ip{\tilde\chm_a(\bR),\bt} = 0
\end{equation}
for all $a\in \cA$, is a solution of $\qle$.
\end{prop}
\begin{proof} The exterior derivative of~\eqref{eq:impl} is
$\ip{\bR^*\d\tilde\chm_a,\bt} + \bR^*\theta_a= 0$, since $\d\bt$ takes values in
  $\t$. Hence, writing $R^a=r_a\circ \bR$ for the components of $\bR$, it
  follows from~\eqref{eq:scalar} and~\eqref{eq:impl} that
\begin{align*}
-R^*\theta_a\wedge\d R^a&= \ip{R^*\d\tilde\chm_a,\bt}\wedge \d R^a\\
&=\bigl(\ip{\tilde\chm_a(\bR),\bt}\beta_a + {\textstyle\sum_{b\in\cA}
R^*\gamma_{ab}\ip{\tilde\chm_b(\bR),\bt}\d R^b}\bigr)\wedge \d R^a=0.
\end{align*}
for all $a\in \cA$.
\end{proof}
In the integrable case, the linear system~\eqref{eq:scalar} is compatible, so
there are many such solutions. Tsarev originally presented these systems in
particularly convenient gauge: under scaling equivalence $\chm_a\mapsto f_a
\chm_a$, ~\eqref{eq:scalar} is modified by adding $f_a^{-1}\d f_a$ to
$\beta_a$ and rescaling $\gam_{ab}$.  Tsarev used this freedom to ensure that
constants are scalar solutions of~\eqref{eq:scalar}.  This forces
$\beta_a=-\sum_b\gamma_{ab}\d r_b$ mod $\d r_a$ and~\eqref{eq:scalar} becomes
\[
\d\chm_a = {\textstyle\sum_{b\in\cA} \gamma_{ab}(\chm_b-\chm_a) \d r_b\mod \d r_a},
\quad\text{i.e.,}\quad \del_b \chm_a = \gamma_{ab}(\chm_b-\chm_a)
\quad\text{for}\quad b\neq a.
\]
Since $\d\beta_a\wedge\d r_a=\sum_{b,c}(\del_c\gamma_{ab})\d r_a\wedge \d
r_b\wedge \d r_c$, the integrability condition~\eqref{eq:hyd-int} becomes
$\del_c\gamma_{ab}=\del_b\gamma_{ac}$ and $\del_b\gamma_{ac}+
\gamma_{ac}\gamma_{cb}-\gamma_{ac}\gamma_{ab}-\gamma_{ab}\gamma_{bc}$ for
$a,b,c$ distinct.

\subsection{Hydrodynamic reductions}\label{sec:hr} One remaining task before
proving the main theorem is to formalize the notion of a hydrodynamic
reduction of a QLS $\qle\leq \t^*\tens T\cod$ on maps $\uk\colon \as\to\cod$
(for a flat affine manifold $\as$). Recall from the introduction that the data
defining an $N$-component hydrodynamic reduction are a map $U\colon\cov\to
\cod$ and maps $\chm_a\colon \cov\to \t^*$ for $a\in\cA=\{1,\ldots N\}$ and
$\cov$ open in $\R^N$ with coordinates $r_1,\ldots r_N\colon \cov\to
\R$. These data are required to satisfy two properties:
\begin{numlist}
\item the hydrodynamic system on maps $\bR\colon \as\to \cov$ defined by
  $\d R^a\wedge \ip{\chm_a(\bR),\d\bt}=0$ (for all $a\in\cA$) is compatible
  (where $R^a=r_a\circ \bR$ are the components of $\bR$);
\item if $\bR$ solves this system, then $\uk=U\circ\bR$ solves $\qle$.
\end{numlist}

Note that the characteristic momenta $\chm_a:a\in\cA$ are only naturally
defined up to scale, and the hydrodynamic system may be rephrased that for all
$a\in \cA$ there are scalar valued functions $f_a(r_1,\ldots r_n)$ (depending
on the solution) such that $\d R^a/\d\bt = f_a(\bR) \chm_a(\bR)$. As in the
introduction, the chain rule for $\uk=\Uk\circ \bR$ implies
\begin{align*}
\frac{\d \uk}{\d\bt} &= \frac{\bR^* \d \Uk\circ \d\bR}{\d\bt}
= \sum_{a\in\cA} \frac{\d R^a}{\d\bt} \tens \del_a \Uk(\bR)
= \sum_{a\in\cA} f_a(\bR) \chm_a(\bR)\tens\del_a\Uk(\bR).
\end{align*}
Under the compatibility of the hydrodynamic system (so that it has many
independent solutions $\bR$), criterion (2) above is therefore equivalent to
the property that $\beta_a:=\chm_a\tens \del_a \Uk$ is in $\qle$ for all
$a\in\cA$. But then, where $\beta_a$ is nonzero, $[\beta_a]$ is in the rank
one variety $\cR^\qle$, i.e., $[\chm_a]\in\chv^\qle$ is characteristic and
$[\del_a \Uk]\in\cC^\qle$ is cocharacteristic. Geometrically, the hydrodynamic
reduction provides many \emph{$N$-secant solutions} $\uk$ to $\qle$, i.e., the
projective image of $\d\uk$ meets the rank one variety $\cR^\qle$ in $N$ points.

\begin{defn} An \emph{$N$-component hydrodynamic reduction} of a QLS
$\qle\leq\t^*\tens T\cod$ with characteristic variety $\chv^\qle$ is a map
\[
(\Uk,[\chm_1],\ldots[\chm_N])\colon
\cov\to \chv^\qle\times_\cod\cdots \times_\cod \chv^\qle,
\]
where $\cov$ is open in $\R^N$ and the codomain is the $N$-fold fibre product,
such that $\chm_a\tens \del_a \Uk$ is in $\qle$ for all $a\in\{1,\ldots N\}$
and the hydrodynamic system with characteristic momenta $\chm_1,\ldots \chm_N$
is compatible as described in~\eqref{eq:scalar}.

Then~\cite{FKh3,FKhn} $\qle$ is \emph{integrable by hydrodynamic reductions} if
for all $N\geq 2$ it admits $N$-component hydrodynamic reductions
parameterized by $N(n-2)$ functions of $1$-variable.
\end{defn}

For $N\leq m=\dim\cod$, a hydrodynamic reduction generically and locally
determines an $N$-dimensional submanifold $\ned$ of $\cod$ (the image of $\Uk$)
together with a net $\del_a \Uk:a\in\cA$ on $\ned$ with $[\del_a
  \Uk]\in\cC^\qle$. Thus a hydrodynamic reduction defines a net satisfying
Definition~\ref{d:cnet}~\ref{i:cnet}. Conversely, if $\qle$ is a compliant
QLS, then for any such net, Proposition~\ref{p:comp} shows that the embedding
of $\cC^\qle$ into $\Proj(\t^*\tens\cV^\qle)$ gives $\del_a
\Uk=\chm_a\tens v_a$ for some local sections $v_a$ of $\Uk^*\cV^\qle$, where
$[\chm_a]$ are the characteristic momenta corresponding to $[\del_a\Uk]$ under
the isomorphism $\zeta^\qle\colon \cC^\qle\to \chv^\qle$.

\begin{proof}[Proof of Theorem~\ref{t:main}] The discussion so far has
estabilished a correpondence between hydrodynamic reductions, modulo the
compatibility condition, and nets satisfying~\ref{d:cnet}~\ref{i:cnet}. It
therefore remains to show that under this correspondence, the compatibility of
the hydrodynamic system is equivalent to~\ref{d:cnet}~\ref{i:cnet2}, i.e., is
automatic if $\mathrm{rank}\,\cV^\qle\geq 2$ and is equivalent to the net
being conjugate otherwise.

The proof of this last step follows the line of argument in~\cite{DFKN1},
cf.~also~\cite{BFT} (QLS of type I) and \cite{FHK} (QLS of type H). First
choose a basis $\eps_1,\ldots \eps_n$ for $\t^*$ and rescale the
characteristic momenta such that $\chm_{a1}=1$.  Under the embedding of
$\cC^\qle$ into $\Proj(\t^*\tens\cV^\qle)$, $\del_b \Uk=\chm_b\tens v_b$ for
some local sections $v_b$ of $\cV^\qle$ over $\ned$, and hence, in $\t^*$
components, $\del_b \Uk_k = \chm_{bk} v_b= \chm_{bk} \del_b \Uk_1$ for
$k\in\{1,\ldots n\}$. Taking the $\del_a$ derivative of this equation and
commuting partial derivatives yields
\begin{equation*}
(\del_a\chm_{bk}) \del_b \Uk_1-(\del_b\chm_{ak}) \del_a \Uk_1
= (\chm_{ak}-\chm_{bk}) \del_a\del_b \Uk_1.
\end{equation*}
On dividing by $\chm_{ak}-\chm_{bk}$, the right hand side is independent of $k$
and hence
\begin{equation*}
\Bigl(\frac{\del_a\chm_{bk}}{\chm_{ak}-\chm_{bk}}
-\frac{\del_a\chm_{b\ell}}{\chm_{a\ell}-\chm_{b\ell}}\Bigr) v_b
= \Bigl(\frac{\del_b\chm_{ak}}{\chm_{ak}-\chm_{bk}}
-\frac{\del_b\chm_{a\ell}}{\chm_{a\ell}-\chm_{b\ell}}\Bigr) v_a.
\end{equation*}
Thus both sides are zero unless $v_a$ and $v_b$ are linearly dependent, i.e.,
multiples of some $v\in \cV^\qle$, say.  But then the span of $\del_a \Uk =
\chm_a\tens v_a$ and $\del_b \Uk = \chm_b\tens v_b$ is
$\mathrm{span}\{\chm_a,\chm_b\}\tens\mathrm{span}\{v\}$, hence decomposable.
For $\mathrm{rank}(\cV^\qle)\geq 2$, the set where this holds has empty
interior by condition (4) of compliancy, and so the hydrodynamic compatibility
criterion is satisfied on the dense complement, hence everywhere by
continuity.

It remains to establish the equivalence in the case
$\mathrm{rank}(\cV^\qle)=1$.

If the compatibility condition
$\del_a\chm_{bk}=\gam_{ba}(\chm_{ak}-\chm_{bk})$ for $a\neq b$ holds, then
\begin{align*}
\del_a\del_b \Uk_k &= (\del_a\chm_{bk}) \del_b \Uk_1+
\chm_{bk}\,\del_a\del_b \Uk_1\\
&=\gam_{ba}(\chm_{ak}-\chm_{bk})\del_b \Uk_1
+ \chm_{bk} (\gam_{ab}\del_a \Uk_1+\gam_{ba}\del_b \Uk_1)\\
&=\gam_{ab}(v_a/v_b) \del_b \Uk_k +\gam_{ba} (v_b/v_a) \del_a \Uk_k.
\end{align*}
Thus $\del_a\del_b \Uk$ is in the span of and $\del_a\Uk$ and $\del_b \Uk$, so
the net is conjugate.

Conversely, if the net is conjugate with $\del_a\del_b \Uk_k=
\alpha_{ab} \del_b \Uk_k +\beta_{ab}\del_a \Uk_k$ for $a\neq b$, then
taking $k=1$,
\[
\del_a\del_b \Uk_1 = \alpha_{ab} \del_b \Uk_1 +\beta_{ab}\del_a \Uk_1
= \alpha_{ab} v_b +\beta_{ab} v_a.
\]
On the other hand, the $\del_a$ derivative of $\del_b \Uk_k = \chm_{bk} \del_b
\Uk_1$ yields
\[
  \chm_{bk}\,\del_a\del_b \Uk_1 = \del_a\del_b \Uk_k -
  (\del_a\chm_{bk})\del_b \Uk_1
  = \alpha_{ab}  \chm_{bk} v_b +\beta_{ab} \chm_{ak} v_a - (\del_a\chm_{bk}) v_b.
\]
Eliminating $\del_a\del_b \Uk_1$ between these equations, it follows that
\begin{equation*}
  \alpha_{ab}\chm_{bk} v_b +\beta_{ab} \chm_{bk} v_a
  = \alpha_{ab}  \chm_{bk} v_b +\beta_{ab} \chm_{ak} v_a - (\del_a\chm_{bk}) v_b
\end{equation*}
and hence $\del_a\chm_{bk} = \beta_{ab}(v_a/v_b) ( \chm_{ak} - \chm_{bk} )$,
which is the compatibility condition.
\end{proof}

\begin{cor} A compliant QLS $\qle\leq\t^*\tens T\cod$ is integrable by
  hydrodynamic reductions if and only if $\cod$ admits a family of
  $3$-dimensional cocharacteristic nets parametrized by $3(n-2)$ functions of
  one variable.
\end{cor}
In particular, this Corollary applies to generic QLS of types G, H and I,
unifying and extending results in~\cite{BFT,DFKN1,FHK}.

\acknowledge

I would like to thank the Eduard \v Cech Institute, grant GA CR P201/12/G028,
for financial support, and Robert Bryant, Jenya Ferapontov, Boris Kruglikov
and Vladimir Souc\v ek for invaluable discussions.

%  Bibliography formatting
%
\newcommand{\noopsort}[1]{}
\newcommand{\bauth}[1]{\mbox{#1}}
\newcommand{\bart}[1]{\textit{#1}}
\newcommand{\bjourn}[4]{#1\ifx{}{#2}\else{ \textbf{#2}}\fi{ (#4)}}
\newcommand{\bbook}[1]{\textsl{#1}}
\newcommand{\bseries}[2]{#1\ifx{}{#2}\else{ \textbf{#2}}\fi}
\newcommand{\bpp}[1]{#1}
\newcommand{\bdate}[1]{ (#1)}
\def\band/{and}


\begin{thebibliography}{10}

\bibitem{AG} M.\,A. Akivis and V.\,V. Goldberg, \textit{Projective differential
geometry of submanifolds}, North-Holland Mathematical Library \textbf{49},
Elsevier, Amsterdan, 1993.

\bibitem{BFKN} S. Berjawi, E. Ferapontov, B. Kruglikov and V. Novikov,
\textit{Second-order PDEs in 3D with Einstein-Weyl conformal structure},
Preprint (2021), arXiv:2104.02716.

\bibitem{BFT} P.\,A. Burovskii, E.\,V. Ferapontov and S.\,P. Tsarev,
\textit{Second order quasilinear PDEs and conformal structures in projective
space}, \bjourn{Int. J. Math.}{21}{}{2010} 799--841.

\bibitem{Cal:ibg} \bauth{D.\,M.\,J. Calderbank},
\textit{Integrable background geometries}, \bjourn{SIGMA}{10}{34}{2014}.

\bibitem{CK} \bauth{D.\,M.\,J. Calderbank and B. Kruglikov},
\textit{Integrability via geometry: dispersionless differential equations in
three and four dimensions}, Preprint (2016), arXiv:1612.02753.

\bibitem{DFKN1} B. Doubrov, E.\,V. Ferapontov, B. Kruglikov and V. Novikov,
{\it On the integrability in Grassmann geometries: integrable systems
associated with fourfolds in \textup{Gr(3,5)}},
\bjourn{Proc. London Math. Soc.}{116}{}{2018} 1269--1300, arXiv:1503.02274.

\bibitem{DFKN2} B. Doubrov, E.\,V. Ferapontov, B. Kruglikov and V. Novikov,
{\it Integrable systems in $4D$ associated with sixfolds in
\textup{Gr(4,6)}}, \bjourn{Int. Math. Res. Notices}{}{}{to appear},
arXiv:1705.06999.

\bibitem{FHK}  E.\,V. Ferapontov, L. Hadjikos and K.\,R. Khusnutdinova,
\textit{Integrable equations of the dispersionless Hirota type and
hypersurfaces in the Lagrangian Grassmannian},
\bjourn{Int. Math. Res. Notices}{}{}{2010} 496--535, arXiv:0705.1774.

\bibitem{FKh3} E.\,V. Ferapontov and K.\,R. Khusnutdinova,
\textit{On the integrability of $(2+1)$-dimensional quasilinear systems},
\bjourn{Comm. Math. Phys.}{248}{}{2004} 187--206.

\bibitem{FKhc2} E.\,V. Ferapontov and K.\,R. Khusnutdinova,
\textit{The characterization of two-component $(2+1)$-dimensional integrable
systems of hydrodynamic type}, \bjourn{J. Phys.}{A 37}{}{2004} 2949--2962.

\bibitem{FKhn}  E.\,V. Ferapontov and K.\,R. Khusnutdinova,
\textit{Hydrodynamic reductions of multi-dimensional dispersionless PDEs:
the test for integrability}, \bjourn{J. Math. Phys.}{45}{}{2004} 2365--2377.

\bibitem{FK1} E.\,V. Ferapontov and B. Kruglikov,
\textit{Dispersionless integrable systems in $3D$ and Einstein--Weyl geometry},
\bjourn{J. Diff. Geom.}{97}{}{2014} 215--254.

\bibitem{FK2} E.\,V. Ferapontov and B. Kruglikov, 
\textit{Dispersionless integrable hierarchies and \textup{GL}$(2,\R)$ geometry},
Preprint (2016), arXiv:1607.01966.

\bibitem{GT} J. Gibbons and S.\,P. Tsarev, \textit{Reductions of the Benney
equations}, \bjourn{Phys. Lett.}{A 211}{}{1996} 19--24.

\bibitem{GH} P. Griffiths and J. Harris, \textit{Principles of algebraic
geometry}, Wiley, New York, 1978.

\bibitem{Lee} J.\,M. Lee, \textit{Manifolds and differential geometry},
Graduate studies in mathematics \textbf{107}, Amer. Math. Soc., Providence, 2009.

\bibitem{Ode} A.\,V. Odesskii, \textit{Integrable structures of dispersionless
systems and differential geometry}, \bjourn{Theor. Math. Phys.}{191}{}{2017}
692--709.

\bibitem{OS} A.\,V. Odesskii and V.\,V. Sokolov, \textit{Integrable
$(2+1)$-dimensional systems of hydrodynamic type},
\bjourn{Theor. Math. Phys.}{163}{}{2010} 549--586.

\bibitem{Smith} A.\,D. Smith,
\textit{Integrable \textup{GL(2)} geometry and hydrodynamic partial differential
equations}, \bjourn{Comm. Anal. Geom.}{18}{2010} 743--790.

\bibitem{Smith1} A.\,D. Smith, \textit{A geometry for second order PDEs and
  their integrability}, Preprint (2010), arXiv:1010.6010.

\bibitem{Smith-talk} A.\,D. Smith,
\textit{Towards generalized hydrodynamic integrability via the characteristic
variety}, Fields Institute, Toronto (2013).

\bibitem{Smith-notes} A.\,D.\ Smith, \textit{Involutive Tableaux,
  Characteristic Varieties, and Rank-one Varieties in the Geometric Study of
  PDEs}, Lecture Notes (2017), arXiv:1701.04930.

\bibitem{Tsarev} S.\,P. Tsarev, \textit{Geometry of hamiltonian systems of
hydrodynamic type. Generalized hodograph method},
\bjourn{Math. USSR Isv.}{54}{}{1990} 1048--1068.

\end{thebibliography}
\end{document}